\documentclass[final, 12pt]{msml2022} 

\title{Error Estimates for the Deep Ritz Method with Boundary Penalty}
\usepackage{times}
\usepackage{thm-restate}

\usepackage{mathtools}

\newcommand{\jm}[1]{\textcolor{black}{#1}}
\definecolor{dkgold}{rgb}{0.639, 0.53, 0.19}
\newcommand{\mz}{\color{black}}
\newcommand{\eee}{\color{black}}

\msmlauthor{%
 \Name{Johannes M\"uller} \Email{jmueller@mis.mpg.de}\\
 \addr Max Planck Institute for Mathematics in the Sciences, Leipzig, Germany
 \AND
 \Name{Marius Zeinhofer} \Email{mariusz@simula.no}\\
 \addr Simula Research Laboratory, Oslo, Norway%
}

\makeatletter
 \let\Ginclude@graphics\@org@Ginclude@graphics
\makeatother

\begin{document}

\maketitle

\begin{abstract}%
  We estimate the error of the Deep Ritz Method for linear elliptic equations. For Dirichlet boundary conditions, we estimate the error when the boundary values are imposed through the boundary penalty method. Our results apply to arbitrary sets of ansatz functions and estimate the error in dependence of the optimization accuracy, the approximation capabilities of the ansatz class and -- in the case of Dirichlet boundary values -- the penalization strength $\lambda$. To the best of our knowledge, our results are presently the only ones in the literature that treat the case of Dirichlet boundary conditions in full generality, i.e., without a lower order term that leads to coercivity on all of $H^1(\Omega)$. Further, we discuss the implications of our results for ansatz classes which are given through ReLU networks and the relation to existing estimates for finite element functions. For high dimensional problems our results show that the favourable approximation capabilities of neural networks for smooth functions are inherited by the Deep Ritz Method.
\end{abstract}


\section{Introduction}\label{sec:introduction}
Utilizing artificial neural networks as ansatz classes for the solution of PDEs has recently gained interest due to its potential for parametric families of PDEs \cite{li2020fourier}, inverse or data enhanced problems \cite{zhang2018deep, zhu2019physics} and the solution of PDEs in high spatial dimensions \cite{han2018solving, han2017deep, jentzen2018proof}. Popular approaches are physics informed neural networks \cite{raissi2019physics}, neural operator methods \cite{li2020fourier} and the Deep Ritz Method \cite{weinan2018deep}. 

Classical numerical schemes for PDEs, such as FEM or FD, are usually designed to solve one PDE instance at a time. In contrast, neural network based approaches allow to simultaneously solve parametric problems. Examples range from parametric geometries to simultaneously solving equations for different material properties or boundary conditions. This allows for a very effective design space exploration and is of great benefit in engineering applications \cite{hennigh2021nvidia}. 

Further, neural network based approaches to inverse problems, i.e., when certain parameters of the PDE are to be inferred from measured data, have shown promising results recently \cite{hennigh2021nvidia, raissi2019physics}. For an effective software realization of the solution of the aforementioned problems we refer to Nvidia Modulus \cite{hennigh2021nvidia}, a software library built on TensorFlow \cite{abadi2016tensorflow} to enable seamless integration in engineering workflows.

Another surge of interest for neural network ansatz classes for the solution of PDEs stems from their ability to approximate high dimensional functions effectively, possibly breaking the curse of dimensionality \cite{weinan2019barron, wojtowytsch2020some, jentzen2018proof}. Consequently, a line of research focuses on the capabilities of neural networks to approximate solutions of high dimensional PDEs breaking or mitigating the curse of dimensionality.

\subsection{Main Contribution and Related Work}
In this manuscript, we investigate the error of the Deep Ritz Method for linear elliptic equations in terms of the optimization quality, the expressiveness of the ansatz class and -- in the case of Dirichlet boundary values -- the penalty parameter. The class of linear elliptic PDEs is pervasive and estimating the error of the Deep Ritz Method for these PDEs constitutes an important first step before considering more complex situations, such as inverse and parametric problems. 

As a model problem we consider the equation posed on $\Omega\subset\mathbb{R}^d$
\begin{align}\label{eq:divagrad}
    \begin{split}
        -\operatorname{div}(A\nabla u) & = f \quad \text{in } \Omega \\
        u & = 0 \quad \text{on } \partial\Omega,
    \end{split}
\end{align}
where $A\in L^\infty(\Omega,\mathbb{R}^{d\times d})$ is symmetric and elliptic and $f\in L^2(\Omega)$, although our method is general enough to handle a larger class of problems. In particular, our analysis covers also non-essential boundary value problems. Using a penalty parameter $\lambda >0$ and neural network functions $u_\theta$, where $\theta\in\Theta$ denotes the parameters of the network, the Deep Ritz formulation of \eqref{eq:divagrad} is the problem of (approximately) finding
\begin{equation}\label{eq:deep_ritz}
    \theta^* \in \underset{\theta \in \Theta}{\operatorname{argmin}}\left[
    \frac12\int_\Omega A\nabla u_\theta\cdot\nabla u_\theta\mathrm dx + \lambda \int_{\partial\Omega}u^2_\theta\mathrm ds - \int_\Omega fu_\theta\mathrm dx \right].
\end{equation}
There are recent results in the literature that estimate the error of the Deep Ritz Method for elliptic equations in certain cases \cite{xu2020finite, jiao2021error, duan2021analysis}. Common in all of them is that they circumvent a major technical difficulty by not considering Dirichlet boundary conditions for general elliptic operators. Usually, Dirichlet boundary conditions are imposed by the boundary penalty method, compare to the formulation \eqref{eq:deep_ritz} and \cite{weinan2018deep}. The works \cite{xu2020finite, jiao2021error, duan2021analysis} consider Dirichlet boundary conditions in the Deep Ritz Method via penalization only in the setting of a coercive operator on all of $H^1(\Omega)$ \footnote{Note that this excludes the prototypical Poisson problem $-\Delta u = f$ with $u_{|\partial\Omega}=0$.}. Thus, the effect of the boundary penalty in the manuscripts \cite{xu2020finite, jiao2021error, duan2021analysis} can immediately be deduced by the results from the classical paper of \cite{maury2009numerical}. One of our main contributions (see Theorem~\ref{thm:main}) is to  extend the results of \cite{maury2009numerical} and to relax the coercivity assumption on the operator in the boundary penalty method, allowing operators that are coercive only on the space $H^1_0(\Omega)$. To the best of our knowledge, our result is presently the only one available in the literature that includes this case. 

We then discuss the error decay rates for the boundary penalty method in dependence of the ansatz class and find that even for optimal penalization they might differ from the approximation rate (see Section~\ref{sec:rates}, Theorem~\ref{thm:rates}). Finally note that in this manuscript, we do not take the numerical evaluations of the appearing integrals into account, however, this can be done exactly as in the paper of \cite{mishra2020estimates}. 

With respect to high dimensional PDEs our results reveal -- when combined with suitable approximation results such as \cite{guhring2019error}, see also Theorem \ref{thm:quantUA} -- that solving elliptic PDEs by the Deep Ritz Method retains the favorable approximation capabilities of neural networks for smooth (in the sense of Sobolev regularity), high dimensional functions. More precisely, the mitigation of the curse of dimensionality through sufficient smoothness of the approximating function can be transferred, also in the case of the boundary penalty method for Dirichlet boundary conditions, compare to Remark \ref{rem:smoothness} and \ref{rem:adaptionToSmoothness2}. We stress that our results do not show a dimension independent error decay. However, this cannot be expected in this generality as we do not assume any special regularity (such as Barron regularity) or lower dimensionality conditions on the data of the PDEs. Note, that also in the important works of \cite{jentzen2018proof} on breaking the curse of dimensionality for certain parabolic PDEs always low-dimensionality assumptions on the initial data and the PDE coefficients are present. In contrast, our results reveal that sufficient (dimension dependent) smoothness can serve as such an assumption for the application of the Deep Ritz Method to elliptic equations. 

Elliptic Equations in high spatial dimensions are important as building blocks for parabolic equations which have well known applications in high dimensions, for instance in finance \cite{han2017deep}. For example, using a minimizing movement scheme, an elliptic problem (with an additional term corresponding to dissipation) needs to be solved for every time step \cite{hwang2021deep} in order to solve a parabolic equation. Further, elliptic equations arise naturally as equilibrium states for parabolic equations.

\section{Preliminaries}\label{sec:preliminaries}

In this section we give an overview of the notations and concepts that we use in the remainder. 

\subsection{Sobolev spaces and Friedrich's inequality}

We denote the space of functions on \(\Omega\subseteq\mathbb R^d\) that are integrable in $p$-th power by \(L^p(\Omega)\) and endow it with its canonical norm. We denote the subspace of \(L^p(\Omega)\) of functions with weak derivatives up to order \(k\) in $L^p(\Omega)$ by \(W^{k,p}(\Omega)\), which is a Banach space with the norm
\[ \lVert u\rVert_{W^{k,p}(\Omega)}^p \coloneqq \sum_{l=0}^k\lVert D^{l} u \rVert_{L^p(\Omega)}^p. \]
This space is called a \emph{Sobolev space} and we denote its dual by $W^{k,p}(\Omega)^\ast$. The closure of all compactly supported smooth functions \(C_c^\infty(\Omega)\) in \(W^{k,p}(\Omega)\) is denoted by \(W^{k,p}_0(\Omega)\).
It is well known that for Lipschitz domains $\Omega$, the operator restricting a Lipschitz continuous function defined on \(\overline{\Omega}\) to the boundary admits a linear and bounded extension $\operatorname{tr}\colon W^{1,p}(\Omega)\to L^p(\partial\Omega)$, called the \emph{trace operator}.
Further, we write \(\lVert u\rVert_{L^p(\partial\Omega)}\) whenever we mean \(\lVert \operatorname{tr}(u)\rVert_{L^p(\partial\Omega)}\). In the following we mostly work with the case $p=2$ and write $H^k_{(0)}(\Omega)$ instead of $W^{k,2}_{(0)}(\Omega)$.

Our analysis of the boundary penalty method relies on the Friedrich inequality which states that the \(L^2(\Omega)\) norm of a function can be estimated by the norm of its gradient and boundary values. 

\begin{proposition}[Friedrich's inequality]\label{poinclemma}
    Let \(\Omega\subseteq\mathbb R^d\) be a bounded and open set with Lipschitz boundary \(\partial\Omega\). Then there exists a constant \(c_{F}>0\) such that
    \begin{equation}\label{Poincarefinal}
    \lVert u\rVert_{H^1(\Omega)} \le c_{F}^2 \cdot \left(\lVert \nabla u\rVert_{L^2(\Omega)}^2 + \lVert u\rVert_{L^2(\partial\Omega)}^2\right) \quad \text{for all } u\in H^1(\Omega).
    \end{equation}
\end{proposition}
In particular this implies that $(\lVert \nabla u\rVert_{L^2(\Omega)}^2 + \lVert u\rVert_{L^2(\partial\Omega)}^2)^{1/2}$ defines an equivalent norm on $H^1(\Omega)$. We refer to the optimal constant \(c_{F}\) as the \emph{Friedrich constant} of \(\Omega\).

\subsection{Preliminaries from neural network theory}

Consider natural numbers \(d, m, L, N_0, \dots, N_L\) and let $\theta = \left((A_1, b_1), \dots, (A_L, b_L)\right)$ be a tupel of matrix-vector pairs where \(A_l\in\mathbb R^{N_{l}\times N_{l-1}}, b_l\in\mathbb R^{N_l}\) and \(N_0 = d, N_L = m\). Every matrix vector pair \((A_l, b_l)\) induces an affine linear map \(T_l\colon \mathbb R^{N_{l-1}} \to\mathbb R^{N_l}\). The \emph{neural network function with parameters} \(\theta\) and with respect to some \emph{activation function} \(\rho\colon\mathbb R\to\mathbb R\) is the function
\[u^\rho_\theta\colon\mathbb R^d\to\mathbb R^m, \quad x\mapsto T_L(\rho(T_{L-1}(\rho(\cdots \rho(T_1(x)))))).\]
We refer to the parametric family of functions
\(\mathcal F^\rho_{\Theta} \coloneqq \left\{ u^\rho_\theta\mid\theta\in\Theta\right\} \)
as a \(\rho\)-network, where \(\Theta\) denotes a family of parameters and drop the superscript \(\rho\) from now on. If we have \(f=u_\theta\) for some \(\theta\in\Theta\) we say the function \(f\) can be \emph{expressed} by the neural network. We call \(d\) the \emph{input} and \(m\) the \emph{output dimension}, \(L\) the \emph{depth} and \(W=\max_{l=0, \dots, L} N_l\) the \emph{width} of the neural network. We call a network \emph{shallow} if it has depth \(2\) and \emph{deep} otherwise. The \emph{number of parameters} and the \emph{number of neurons} of such a network is given by \(\operatorname{dim}(\Theta)\) and \(\sum_{l=0}^L N_l\). In the remainder, we restrict ourselves to the case \(m=1\) since we only consider real valued functions.

The \emph{rectified linear unit} or \emph{ReLU activation function} is defined via \(x\mapsto \max\left\{ 0, x\right\}\) and we call networks with this particular activation function \emph{ReLU networks}. The class of ReLU networks coincides with the class of continuous and piecewise linear functions, which yields the following result \cite{arora2016understanding, dondl2021uniform}.

\begin{theorem}[Universal approximation with zero boundary values]\label{UniversalApproximation}
Consider an open set \(\Omega\subseteq\mathbb R^d\) and \(u\in W^{1,p}_0(\Omega)\) with $p\in[1,\infty)$. Then for all \(\varepsilon>0\) there exists \(u_\varepsilon\in W^{1,p}_0(\Omega)\) that can be expressed by a ReLU network of depth \(\lceil \log_2(d+1)\rceil +1\) such that $\left\lVert u - u_\varepsilon \right\rVert_{W^{1,p}(\Omega)}\le \varepsilon$.
\end{theorem}

The difference of this result to other universal approximation results \cite{hornik1991approximation, czarnecki2017sobolev} is the approximating neural network function are guaranteed to have zero boundary values. This is a special property of the ReLU activation function and implies the consistency of the boundary penalty method for arbitrary penalization strengths as we will see later. In order to quantify the error that is being made by the variational training of ReLU networks with boundary penalty, we use the following result from \cite{guhring2019error}, where other results on approximation bounds in Sobolev spaces have been obtained in~\cite{guhring2020approximation, xu2020finite, siegel2020approximation, siegel2020high, hon2021simultaneous, duan2021convergence, de2021approximation}.
\begin{restatable}[Quantitative universal approximation]{theo}{quantUA}\label{thm:quantUA}
Let \(\Omega\subseteq\mathbb R^d\) be a bounded and open set with Lipschitz regular boundary\footnote{or more generally, that \(\Omega\) is a Sobolev extension domain}, let \(k\in(1, \infty)\), \(p\in[1, \infty]\) and fix an arbitrary function $u\in W^{k,p}(\Omega)$. Then, for every \(n\in\mathbb N\), there is a ReLU network $u_n$ with \(\mathcal O(\log_2^2(n^{k/d})\cdot n)\) many parameters and neurons such that
    \[ \left\lVert u - u_n \right\rVert_{W^{s, p}(\Omega)}\le c(s)\cdot \lVert u \rVert_{W^{k, p}(\Omega)} \cdot n^{-(k-s)/d} \]
for every \(s\in[0, 1]\). 
\end{restatable}

\section{Error estimates for the boundary penalty method}\label{sec:errorEstimates}

In this section we bound the distance of a function \(v\in V\) to the solution \(u^\ast\) of the variational problem in dependence of the approximation capabilities of the ansatz class \(V\), the optimization quality and the penalization strength. We begin by stating a general curvature based bound which can be seen as a nonlinear version of C\'{e}a's Lemma.

\subsection{A C\'{e}a Lemma}

The following proof of C\'ea's Lemma is based on the curvature properties of a quadratic, coercive energy defined on a Hilbert space. Note that in the following proposition, $V$ does not need to be a vector space.
\begin{proposition}[C\'ea's Lemma]\label{prop:cea}
    Let $X$ be a Hilbert space, $V\subseteq X$ any subset and $a\colon X\times X \to \mathbb{R}$ a symmetric, continuous and $\alpha$-coercive bilinear form. For $f\in X^*$ define the quadratic energy \(E(u)\coloneqq \frac12a(u,u) - f(u)\) 
and denote its unique minimiser by $u^\ast$. Then for every \(v\in V\) it holds that 
\begin{equation*}
    \lVert v - u^\ast \rVert_X \leq \sqrt{ \frac{2\delta}{\alpha} + \frac{1}{\alpha} \inf_{\tilde v\in V}\lVert \tilde v - u^\ast \rVert_a^2},
\end{equation*}
where \(\delta = E(v) - \inf_{\tilde v\in V}E(\tilde v)\) and \(\lVert u \rVert_a^2\coloneqq a(u, u)\) is the norm induced by \(a\).
\end{proposition}
\begin{proof}
    As $E$ is quadratic it can be exactly expanded using Taylor's formula. Hence, for every $h\in X$ it holds that
    \begin{align*}
    E(u + h) & = E(u^\ast) + \frac12 D^2E(u^\ast)(h,h) = E(u^\ast) + \frac12 a(h, h) = E(u^\ast) + \frac12 \lVert h \rVert_a^2,
    \end{align*}
    where we used \(DE(u^\ast)=0\). Inserting $v - u^\ast$ for $h$ we obtain
    \begin{equation*}
        E(v) - E(u^\ast) = \frac{1}{2} a(v - u^\ast, v - u^\ast) \geq \frac{\alpha}{2}\lVert v - u^\ast \rVert^2_X.
    \end{equation*}
    On the other hand we compute
    \begin{align*}
        E(v) - E(u^\ast) &= E(v) - \inf_{\tilde v\in V}E(\tilde v) + \inf_{\tilde v\in V}(E(\tilde v) - E(u^\ast))
        \\
        &= \delta + \frac12\inf_{\tilde v\in V} \lVert \tilde v - u^\ast \rVert_a^2.
    \end{align*}
    Combining the two estimates and rearranging terms yields the assertion.
\end{proof}

Proposition~\ref{prop:cea} is all we need to derive error estimates for coercive problems with non-essential boundary conditions. We give an example.
\begin{corollary}[Neumann Problem]\label{corollary:neumann_problem}
    Let $\Omega\subset\mathbb{R}^d$ be a bounded Lipschitz domain and let $f$ be a fixed member of $H^1(\Omega)^*$. Denote by $u\in H^1(\Omega)$ the weak solution to 
    \begin{equation*}
        -\Delta u + u = f \quad\text{in }H^1(\Omega)^*.
    \end{equation*}
    Let $\Theta$ be the parameter set of a neural network architecture such that $u_\theta \in H^1(\Omega)$ for every $\theta\in \Theta$. Then for every $\theta\in\Theta$ it holds 
    \begin{equation*}
        \lVert u_\theta - u \rVert_{H^1(\Omega)} \leq \sqrt{2\delta + \inf_{\eta\in\Theta}\lVert u_\eta - u \rVert^2_{H^1(\Omega)}}
    \end{equation*}
    where \[\delta = \lVert u_\theta \rVert^2_{H^1(\Omega)} - f(u_\theta) -  \inf_{\eta\in\Theta}\left[ \lVert u_\eta \rVert^2_{H^1(\Omega)} - f(u_\eta)\right]. \]
\end{corollary}
\begin{proof}
    The bilinear form corresponding to the above Neumann problem is
    \begin{equation*}
        a\colon H^1(\Omega)\times H^1(\Omega) \to \mathbb{R}, \quad a(u,v) = \int_\Omega \nabla u\nabla v + uv\mathrm dx 
    \end{equation*}
    and therefore its coercivity constant is $\alpha=1$ and the associated norm $\lVert\cdot\rVert_a$ is the natural one on $H^1(\Omega)$. Employing Proposition~\ref{prop:cea} yields the assertion.
\end{proof}
\begin{remark}
    Corollary~\ref{corollary:neumann_problem} yields $H^1(\Omega)$ convergence of the Deep Ritz Method provided the ansatz class possesses universal approximation properties with respect to the $H^1(\Omega)$ norm. This is of course also a necessary requirement and fulfilled by a wide class of network architectures and activation functions, see~\cite{hornik1991approximation,czarnecki2017sobolev} or for approximation rates~\cite{guhring2019error}. We stress that any (quantitative) universal approximation theorem for Sobolev topologies can be combined with the above result, such as Theorem \ref{thm:quantUA} for ReLU neural networks. 
    
    Furthermore, the form of the differential equation in the above corollary can easily be generalized. One can for example consider general second order elliptic PDEs in divergence form with non-essential boundary conditions as long as these are coercive and can be derived from a minimization principle.
\end{remark}

\begin{remark}[Dimension Dependence and Adaptation to Smoothness]\label{rem:smoothness}
    Assume the solution $u$ to the Neumann problem is a member of $H^k(\Omega)$ for some $k>1$. Then applying the quantitative universal approximation Theorem~\ref{thm:quantUA} we estimate
    \begin{equation*}
        \lVert u_\theta - u \rVert_{H^1(\Omega)} \precsim \sqrt{ 2\delta + c \cdot \lVert u \rVert_{H^k(\Omega)}n^{-(k-1)/d} },
    \end{equation*}
    where $d\in\mathbb{N}$ is the spatial dimension. 
    While this estimate is not dimension independent, it indicates how smoothness mitigates the deterioration of error decay rates for high dimensions. 
    We see that the merit of neural networks to achieve approximation rates increasing with the smoothness of the target function carries over to the error decay in the deep Ritz method. In contrast, to achieve the approximation rate and an error decay rate of $(k-1)/d$ with finite elements one needs to for example use $P^{k-1}$ elements~\cite{ern2004theory}, which complicates the ansatz class and therefore the approach. 
\end{remark}
\mz 
    \begin{remark}[Practical Realization of Rates]\label{remark:rates_in_practice}
        There is a gap between the theory and the practice of neural network based methods for the solution of PDEs. Error decay rates, as predicted by our results cannot be observed in practice due to the difficulties of computing minima of non-convex functions. In practice, one observes moderate errors that don't decrease beyond a certain accuracy when the numbers of the parameters of the neural network ansatz architecture are increased. We refer to \cite{lee2021partition} and the references therein for a more detailed description of this phenomenon. 
   
       However, what one observes is that neural network based methods in general and the Deep Ritz Method in particular are well suited for problems in high spatial dimensions or a high dimensional parameter space. Practical evidence can already be found in the paper introducing the Deep Ritz Method, see \cite{weinan2018deep} and further high dimensional examples -- even of industrial scale -- can be found in \cite{hennigh2021nvidia}. We propose to view our results as a qualitative explanation of these observations. 
    \end{remark}
\eee

\subsection{An error estimate for the boundary penalty method}
The treatment of Dirichlet boundary conditions corresponds to a constrained optimization problem, as in standard neural network architectures zero boundary values cannot be directly encoded. We use the boundary penalty method as a way to enforce Dirichlet boundary conditions. For ease of presentation, we discuss our approach for the concrete equation
\begin{align}\label{equation:prototype}
    \begin{split}
        -\operatorname{div}\left(A\nabla u\right) & = f \quad \text{in } \Omega \\
        u & = 0 \quad \text{on } \partial\Omega,
    \end{split}
\end{align}
where $A\in L^\infty(\Omega,\mathbb{R}^{d\times d})$ is a symmetric and elliptic coefficient matrix. The weak formulation of this equation gives rise to the bilinear form
\begin{equation*}
    a\colon H^1(\Omega)\times H^1(\Omega) \to \mathbb{R}, \quad a(u,v) = \int_\Omega A\nabla u\cdot\nabla v\mathrm dx
\end{equation*}
and the energy
\begin{equation*}
    E\colon H^1(\Omega) \to \mathbb{R}, \quad E(u) = \frac12 a(u,u) - f(u)
\end{equation*}
where $f\in H^1(\Omega)^*$. Using the boundary penalty method as an approximation for \eqref{equation:prototype} leads to the bilinear form
\begin{equation*}
    a_\lambda\colon H^1(\Omega)\times H^1(\Omega) \to \mathbb{R}, \quad a_\lambda(u,v) = \int_\Omega A\nabla u\nabla v\mathrm dx + \lambda \int_{\partial\Omega}uv\mathrm ds
\end{equation*}
for a penalty parameter $\lambda >0$ and the energy
\begin{equation*}
    E_\lambda\colon H^1(\Omega)\to \mathbb{R}, \quad E_\lambda(u) = \frac12 a_\lambda(u,u) - f(u).
\end{equation*}
The central error estimation is collected in the following Theorem. Note that we require $H^2(\Omega)$ regularity of the solution to equation \eqref{equation:prototype}.

\begin{theorem}\label{thm:main}
Let $\Omega\subset\mathbb{R}^d$ be a bounded domain with $C^{1,1}$ boundary, $f\in L^2(\Omega)$ and assume $A\in C^{0, 1}(\Omega,\mathbb{R}^{d\times d})$ is symmetric, uniformly elliptic with ellipticity constant $\alpha>0$. 
By $u^\ast\in H^1_0(\Omega)$ we denote the solution of~\eqref{equation:prototype} and by $u_\lambda$ the minimizer of the penalized energy $E_\lambda$ over $H^1(\Omega)$. Fix an arbitrary subset $V\subset H^1(\Omega)$ and denote the coercivity constants of $a_\lambda$ by $\alpha_\lambda>0$ and set $\delta \coloneqq E_\lambda(v) - \inf_{\tilde v\in V}E_\lambda(\tilde v)$. Then there is a constant $c>0$, only depending on $A$ and $\Omega$, such that for every \(v\in V\) and \(\lambda > 0\) it holds that 
\begin{equation}\label{eq:generalEstimate}
    \lVert v - u^\ast \rVert_{H^1(\Omega)} \leq \sqrt{ \frac{2\delta}{\alpha_\lambda} + \frac{1}{\alpha_\lambda} \inf_{\tilde v\in V}\lVert \tilde v - u_\lambda \rVert_{a_\lambda}^2} + c  \lambda^{-1}\lVert f \rVert_{L^2(\Omega)},
\end{equation}
where \(\lVert u \rVert_{a_\lambda}^2\coloneqq a_\lambda(u, u)\) is the norm induced by \(a_\lambda\). Further, we can choose 
    \[ c \coloneqq c_F c_{\textit{reg}}\sqrt{\lVert a_1 \rVert} \lVert T \rVert_{\mathcal{L}(H^2(\Omega);\mathcal{H}(\Omega))}, \]
where $T\colon H^2(\Omega)\to H^1(\Omega)$ maps a function $u$ to the $A$-harmonic extension of\footnote{here $\partial_A u = \nu\cdot A\nabla u$ and $\nu$ denotes the outer normal of $\Omega$} $\partial_A u$, $c_F$ denotes the Friedrich constant (see Proposition~\ref{poinclemma}) and $c_{\textit{reg}}$ is the operator norm of 
    \[\left(-\operatorname{div}(A\nabla\cdot)\right)^{-1}\!\colon L^2(\Omega)\to H^2(\Omega)\cap H^1_0(\Omega).\]
\end{theorem}
\begin{proof}
The central idea of the proof consists of the error decomposition
\begin{equation*}
    \lVert v - u^\ast \rVert_{H^1(\Omega)} \leq \lVert v - u_\lambda \rVert_{H^1(\Omega)} + \lVert u_\lambda - u^\ast \rVert_{H^1(\Omega)}.
\end{equation*}
The first error can be treated using C\'ea's Lemma. Note that $a_\lambda$ is in fact coercive on $H^1(\Omega)$ which is a consequence of Friedrich's inequality, see Proposition~\ref{poinclemma}. For the second term one uses a Fourier series expansion in a Steklov basis. The latter is useful for weakly $A$-harmonic functions, hence we investigate the equation satisfied by $v_\lambda\coloneqq u^\ast-u_\lambda$. Due to the regularity assumption on $\Omega$ and $A$ we have $\operatorname{div}(A\nabla u^\ast) \in L^2(\Omega)$ and may integrate by parts to obtain for all $\varphi\in H^1(\Omega)$
\begin{equation}
    \int_\Omega f\varphi\mathrm dx = - \int_\Omega \operatorname{div}(A\nabla u^\ast)\varphi\mathrm dx = \int_\Omega A\nabla u^\ast\nabla\varphi\mathrm dx - \int_{\partial\Omega}\partial_Au^\ast\varphi\mathrm ds.\label{equation:eq_on_all_H1}
\end{equation}
Using the optimality condition of $u_\lambda$ yields
\[
    \int_\Omega (A\nabla u_\lambda)\cdot\nabla\varphi\mathrm dx + \lambda\int_{\partial\Omega} u_\lambda\varphi \mathrm ds = \int_\Omega f\varphi \mathrm dx \quad \forall \varphi\in H^1(\Omega).
\]
Subtracting these two equations we obtain that $v_\lambda$ satisfies
\[
    \int_\Omega (A\nabla v_\lambda)\cdot\nabla\varphi\mathrm dx + \int_{\partial\Omega} (\lambda v_\lambda - \partial_A u^\ast)\varphi \mathrm ds = 0 \quad \forall \varphi\in H^1(\Omega).
\]
This implies that $v_\lambda$ is weakly $A$-harmonic, i.e.,  
\[ \int_\Omega (A\nabla v_\lambda)\cdot\nabla\varphi\mathrm dx = 0 \quad \forall \varphi\in H^1_0(\Omega), \]
We claim that there exists a basis $(e_j)_{j\in\mathbb{N}}$ of the space of weakly $A$-harmonic functions and that $v_\lambda$ can be written in terms of this basis as
\begin{equation}\label{equation:series_expansion_v_lambda}
    v_\lambda = \frac1\lambda \sum_{j=0}^\infty c(\lambda)_je_j
\end{equation}
for suitable coefficients $c(\lambda)_j\in\mathbb{R}$. Further, we claim that this Fourier expansion leads to the estimate
\begin{equation}\label{eq:estimateDifference}
    \lVert v_\lambda \rVert_{H^1(\Omega)} \leq \frac c\lambda \lVert f \rVert_{L^2(\Omega)}
\end{equation}
with $c$ as specified in the statement of the Theorem which finishes the proof. The remaining details are provided in the following Section.
\end{proof}
We presented the proof in its above form to draw attention to its key elements and to discuss possible limitations and generalizations.
\begin{remark}[Limitations]
    Our proof requires crucially the $H^2(\Omega)$ regularity of the solution $u^\ast $ to the Dirichlet problem. This is in contrast with the error estimates for non-essential boundary value problems that do not require additional regularity. 
\end{remark} 
\begin{remark}[Generalizations]
    The strategy of the proof of Theorem~\ref{thm:main} holds for a broader class of elliptic zero boundary value problems. The essential requirement is that the bilinear form $a$ of the differential operator is coercive on $H^1_0(\Omega)$ and that $a_\lambda$ is coercive on all of $H^1(\Omega)$. Then, regularity of the solution $u^\ast $ of the zero boundary value problem is required to identify the equation $u^\ast $ satisfies when tested with functions in $H^1(\Omega)$ and not only $H^1_0(\Omega)$, see~\eqref{equation:eq_on_all_H1}.
\end{remark}
\begin{remark}[Optimality of the rate $\lambda^{-1}$]
We demonstrate that the rate\footnote{We write $\precsim$ and $\succsim$ if the inequality $\le$ or $\ge$ holds up to a constant; if both $\precsim$ and $\succsim$ hold, we write $\sim$.}
\[
    \lVert u_\lambda - u^\ast  \rVert_{H^1(\Omega)} \precsim \lambda^{-1}
\]
can not in general be improved. To this end we consider the concrete example
\[
    a_\lambda\colon H^1(0,1)^2 \to \mathbb{R},\quad (u,v)\mapsto \int_0^1 u'v'\mathrm dx + \lambda( u(0)v(0) + u(1)v(1) ). 
\]
The minimiser of $E_\lambda$ with $f\equiv 1$ solves the ODE
\[ 
    -u'' = 1 \quad \text{in }(0,1)
\]
with Robin boundary conditions 
\begin{align*}
    -u'(0) + \lambda u(0) &= 0
    \\
    u'(1) + \lambda u(1) &= 0.
\end{align*}
Its solution is given by 
\[
    u_\lambda(x) = -\frac12 x^2 + \frac12 x + \frac{1}{2\lambda}.
\]
On the other hand the associated Dirichlet problem is solving the same ODE subject to $u(0) = u(1) = 0$ and has the solution
\[  
    u^\ast (x) = -\frac12 x^2 + \frac12 x.
\]
Consequently the difference $u_\lambda - u^\ast $ measured in $H^1(0,1)$ norm is precisely $\frac{1}{2\lambda}$.
\end{remark}

\subsubsection{A solution formula based on Steklov eigenfunctions} 

The Steklov theory yields the existence of an orthonormal eigenbasis of the space  
\begin{equation*}
    \mathcal{H}(\Omega) \coloneqq \big\{ w\in H^1(\Omega) \mid a(w,v) = 0\text{ for all } v\in H^1_0(\Omega) \big\}.
\end{equation*}
of weakly $a$-harmonic functions which we can use for a Fourier expansion of \(v_\lambda\) in order to obtain the desired estimate. For a recent and more general discussion of Steklov theory we refer to \cite{auchmuty2012bases}. The Steklov eigenvalue problem consists of finding $(\mu,w)\in \mathbb{R}\times H^1(\Omega)$ such that
\begin{equation}
    a(w,\varphi) = \mu \int_{\partial\Omega} w \varphi \mathrm ds \quad \text{for all }\varphi\in H^1(\Omega).
\end{equation}
We call \(\mu\) a \emph{Steklov eigenvalue} and \(w\) a corresponding \emph{Steklov eigenfunction}. 

\begin{lemma}[Orthogonal decomposition]\label{orthogonalDecomopsiton}
    We can decompose the space $H^1(\Omega)$ into
    \begin{equation*}
        H^1(\Omega) = \mathcal{H}(\Omega) \oplus_{a_1}H^1_0(\Omega)
    \end{equation*}
    with the decomposition being $a_1$-orthogonal.
\end{lemma}
\begin{proof}
        By the definition of $a_1$ and $\mathcal{H}(\Omega)$ the two spaces are $a_1$ orthogonal. To see that it spans all of $H^1(\Omega)$ let $u \in H^1(\Omega)$ be given. Let \(u_{a}\) be the unique solution of $a(u_{a},\cdot) = 0$ in $H^1_0(\Omega)^*$ subject to $\operatorname{tr}(u_{a}) = \operatorname{tr}(u)$. Then the decomposition is given as $u = u_{a} + (u - u_{a}) = u_{a} + u^\ast $.
    \end{proof}

\begin{theorem}[Steklov spectral theorem]\label{steklovSpectralTheorem}
    Let $\Omega\subseteq\mathbb R^d$ be open and \(a\) be a positive semi-definite bilinear form on $H^1(\Omega)$ and \(\mathcal H(\Omega)\hookrightarrow L^2(\partial\Omega)\) be compact. Then there exists a non decreasing sequence $(\mu_j)_{j \in \mathbb{N}} \subseteq [0,\infty)$ with $\mu_j \to \infty$ and a sequence $(e_j)_{j \in \mathbb{N}} \subseteq \mathcal{H}(\Omega)$ such that $\mu_j$ is a Steklov eigenvalue with eigenfunction $e_j$. Further, $(e_j)_{j\in\mathbb{N}}$ is a complete orthonormal system in \(\mathcal H(\Omega)\) with respect to \(a_1\).
\end{theorem}
\begin{proof}
    This can be derived from the spectral theory for compact operators as for example described in \cite[Section 8.10]{dobrowolski2010angewandte}. In the notation of \cite{dobrowolski2010angewandte}, set $X = \mathcal{H}(\Omega)$ with inner product $a_1$ and let $Y = L^2(\partial\Omega)$ equipped with its natural inner product. Then this yields a divergent sequence $0 < \tilde{\mu}_1 \leq \tilde{\mu}_2\leq \dots$ growing to \(\infty\) and $(e_j)_{j\in\mathbb{N}}\subseteq \mathcal{H}(\Omega)$ with $a_1(e_i,e_j) = \delta_{ij}$ and 
        \begin{equation}\label{auxiliaryEigenvalueEquation}
            a_1(e_j,w) = \tilde{\mu}_j\int_{\partial\Omega}e_jw\mathrm ds \quad \text{for all } w\in\mathcal{H}(\Omega).
        \end{equation}
        For \(\varphi\in H^1(\Omega)\), let $\varphi = \varphi_a + \varphi_0$ be the orthogonal decomposition of the preceding lemma and compute
        \begin{equation*}
            a_1(e_j, \varphi) = a_1(e_j, \varphi_0) + a_1(e_j, \varphi_a) = a_1(e_j, \varphi_a) = \tilde{\mu}_j\int_{\partial\Omega}e_j \varphi_a\mathrm ds = \tilde{\mu}_j\int_{\partial\Omega}e_j \varphi\mathrm ds.
        \end{equation*}
        Using the definition of \(a_1\) we obtain
        \begin{equation*}
            a(e_j,\varphi) = (\tilde{\mu}_j - 1)\int_{\partial\Omega}e_j\varphi\mathrm ds \quad \text{for all } \varphi\in H^1(\Omega).
        \end{equation*}
        Setting $\mu_j \coloneqq \tilde{\mu}_j - 1$ and noting that the above equality implies $\mu_j \geq 0$ concludes the proof.
    \end{proof}

As a direct consequence we obtain the following representation formula.

\begin{corollary}[Fourier expansion in the Steklov eigenbasis]
Let \(w\in \mathcal{H}(\Omega)\). Then we have
    \[ w = \sum_{j=0}^\infty c_j e_j, \]
where 
    \begin{equation}\label{eq:fourier}
        c_j = (1 + \mu_j) \int_{\partial\Omega} we_j \mathrm ds.
    \end{equation}
\end{corollary}
\begin{proof}
Using that \(e_j\) is a Steklov eigenvector, we can compute the Fourier coefficients 
    \[   c_j = a_1(w, e_j) = a(w, e_j) + \int_{\partial\Omega} w e_j\mathrm ds = (1+\mu_j) \int_{\partial\Omega} we_j\mathrm ds. \]
\end{proof}

\begin{lemma}[Solution formula]\label{lemma:solutionformula}
Let $v_\lambda\in H^1(\Omega)$ be the unique solution of
\begin{equation}\label{eq2}
    a(v_\lambda, \varphi) + \int_{\partial\Omega}(\lambda v_\lambda - \partial_A u^\ast )\varphi\,ds = 0 \quad\text{for all }\varphi \in H^1(\Omega).
\end{equation}
Then we have
    \[ v_\lambda = \frac{1}{\lambda} \sum_{j=0}^\infty c(\lambda)_j e_j, \]
where 
    \[ c(\lambda)_j = \frac{1 + \mu_j}{1 + \frac{\mu_j}{\lambda}} \cdot \int_{\partial\Omega} (\partial_A u^\ast )e_j \mathrm ds. \]
\end{lemma}
\begin{proof}
    Note that \(v_\lambda\) is weakly harmonic and hence, we can apply the previous corollary to compute the Fourier coefficients of \(v_\lambda\). Using that \(v_\lambda\) solves \eqref{eq2} and that \(e_j\) is a Steklov eigenfunction we compute
    \begin{align*}
        \int_{\partial\Omega} v_\lambda e_j \mathrm ds = \frac{1}{\lambda} \int_{\partial\Omega} (\partial_A u^\ast )e_j \mathrm ds  - \frac{1}{\lambda} a(v_\lambda, e_j) = \frac{1}{\lambda} \int_{\partial\Omega} (\partial_A u^\ast )e_j \mathrm ds - \frac{\mu_j}{\lambda} \int_{\partial\Omega} v_\lambda e_j \mathrm ds.
    \end{align*}
    Rearranging this yields the following equation which completes the proof
    \[ \int_{\partial\Omega} v_\lambda e_j \mathrm ds = \frac{1}{\lambda}\cdot \frac{1}{1 + \frac{\mu_j}{\lambda}} \int_{\partial\Omega} (\partial_A u^\ast )e_j \mathrm ds. \]
\end{proof}

\subsubsection{Completing the Proof of Theorem~\ref{thm:main}}
We use the explicit solution formula from Lemma~\ref{lemma:solutionformula} for $v_\lambda = u^\ast  - u_\lambda$ to provide the missing claims in the proof of Theorem~\ref{thm:main}.
\begin{proof}[Completing the Proof of Theorem~\ref{thm:main}]
We have already convinced ourselves that \(v_\lambda \coloneqq u_\lambda - u^\ast \) indeed solves \eqref{eq2}. By the means of Lemma~\ref{lemma:solutionformula} it suffices to bound
    \[ \bigg\lVert \sum_{j=0}^\infty c(\lambda)_j e_j \bigg\rVert_{H^1(\Omega)} \]
independently of \(\lambda > 0\). 
Let us denote the $A$-harmonic extension of $\partial_A u^\ast $ with $w$, we obtain
    \[ c(\lambda)_j^2 \le (1 + \mu_j)^2 \left(\int_{\partial\Omega} (\partial_A u^\ast ) e_j \mathrm ds \right)^2 = a_1(w, e_j)^2. \]
    Now we can estimate
    \begin{align*}
        \sum_{j = 0}^\infty (1 + \mu_j)^2 \left(\int_{\partial\Omega} (\partial_A u^\ast ) e_j \mathrm ds \right)^2  & = \sum_{j = 0}^\infty a_1(w, e_j)^2 = a_1(w, w) \\ & \le \lVert a_1 \rVert \, \lVert w \rVert^2_{H^1(\Omega)}
        \\ &\le 
        \lVert a_1  \rVert \, \lVert u^\ast  \rVert^2_{H^2(\Omega)}\lVert T \rVert^2_{\mathcal{L}(H^2(\Omega);H^1(\Omega))}
        \\ & \le
        \lVert a_1  \rVert \,  c^2_{\textit{reg}} \lVert f \rVert^2_{L^2(\Omega)}\lVert T \rVert^2_{\mathcal{L}(H^2(\Omega);H^1(\Omega))},
    \end{align*}
    where $T\colon H^2\to H^1$ is the mapping that assigns a function $u$ the harmonic extension of $\partial_A u$.
    Consequently, we obtain by Parseval's identity
    \begin{equation*}
        \lVert v_\lambda \rVert_{H^1(\Omega)} \le c_F \lVert v_\lambda \rVert_{a_1} = \frac{c_F}{\lambda}\sqrt{ \sum_{j = 1}^\infty c(\lambda)_j^2 } \le \frac{c_F}{\lambda}\sqrt{\lVert a_1 \rVert} \, c_{\textit{reg}}\lVert f \rVert_{L^2(\Omega)}\lVert T \rVert_{\mathcal{L}(H^2(\Omega);H^1(\Omega))}
    \end{equation*}
    which finishes the proof.
\end{proof}

\subsection{Estimates under lower regularity of the right-hand side}
If $f\notin L^2(\Omega)$, Theorem~\ref{thm:main} cannot be applied as the estimation of the term $\lVert u_\lambda - u^\ast  \rVert_{H^1(\Omega)}$ requires that $u^\ast $ is a member of $H^2(\Omega)$. The next Lemma shows that at the expense of a worse rate and norm, we can still estimate this term for distributional right-hand sides $f\in H^1(\Omega)^*$.

\begin{lemma}
    Let $\Omega \subset \mathbb{R}^d$ be a bounded domain with $C^{1,1}$ boundary, assume that $A\in C^{0,1}(\Omega,\mathbb{R}^{d\times d})$ is symmetric and uniformly elliptic, $f\in H^1(\Omega)^*$ and let $u^\ast $ and $u_\lambda$ be as in Theorem~\ref{thm:main}. Then it holds
    \begin{equation*}
        \lVert u_\lambda - u^\ast  \rVert_{L^2(\Omega)} \leq c\cdot\lVert f\rVert_{H^1(\Omega)^*}\lambda^{-1/2}.
    \end{equation*}
\end{lemma}
\begin{proof}
    We set $v_\lambda = u_\lambda - u^\ast $ and denote by $w\in H^1_0(\Omega)$ the solution to the equation $-\operatorname{div}(A\nabla w)=v_\lambda$ in $H^1_0(\Omega)^*$. Then, by our assumptions on $\Omega$ and $A$, the function $w$ is a member of $H^2(\Omega)$. This yields upon integration by parts and the fact that $v_\lambda$ is weakly $A$-harmonic that it holds
    \begin{align*}
        \lVert v_\lambda \rVert^2_{L^2(\Omega)} = \int_\Omega A\nabla w\cdot\nabla v_\lambda\mathrm dx - \int_{\partial\Omega}\partial_A w v_\lambda\mathrm ds 
        &= -\int_{\partial\Omega}\partial_A w u_\lambda \mathrm ds 
        \\&
        \leq c \lVert w \rVert_{H^2(\Omega)}\lVert u_\lambda \rVert_{L^2(\partial\Omega)}
        \\&
        \leq c \lVert v_\lambda \rVert_{L^2(\Omega)}\lVert u_\lambda \rVert_{L^2(\partial\Omega)}.
    \end{align*}
    It remains to estimate $\lVert u_\lambda \rVert_{L^2(\partial\Omega)}$. Note that $u_\lambda$ satisfies
    \begin{equation*}
        0 = \int_\Omega A\nabla u_\lambda\cdot \nabla u_\lambda\mathrm dx + \lambda \int_{\partial\Omega}u_\lambda^2\mathrm ds - f(u_\lambda).
    \end{equation*}
    We get after rearranging and using Young's inequality
       \begin{align*}
            \lVert u_\lambda \rVert^2_{L^2(\partial\Omega)}
            &=
            \frac{2}{\lambda}\left( f(u_\lambda) - \left( \int_\Omega A\nabla u_\lambda\cdot \nabla u_\lambda\mathrm dx + \frac{\lambda}{2}\lVert u_\lambda \rVert_{L^2(\partial\Omega)} \right) \right)
            \\&
            \leq 
            \frac{2}{\lambda}\left( \lVert f \rVert_{H^1(\Omega)^*}\lVert u_\lambda \rVert_{H^1(\Omega)} - \alpha_{\lambda/2}\lVert u_\lambda \rVert_{H^1(\Omega)}^2 \right)
            \\
            &\leq
            \frac{\lVert f \rVert_{H^1(\Omega)^*}^2}{2\alpha_{\lambda/2}\lambda}
        \end{align*}
        which completes the proof.
\end{proof}

\section{Penalization strength and error decay}\label{sec:rates}
We have seen that the distance of an ansatz function can be bounded in terms of the optimization error, the approximation power of the ansatz class and the penalization strength. In this section we discuss the trade off of choosing the penalization strength $\lambda$ too large or too small and discuss the implications of different scalings of $\lambda$ in dependecy of the approximation capabilities of the ansatz classes. We combine our general discussion with Theorem~\ref{thm:quantUA} to obtain Theorem~\ref{thm:rates}, however, our discussion can be combined with any result guaranteeing approximation rates of a sequence of ansatz classes. 

We consider a sequence $(V_n)_{n\in\mathbb N}\subseteq H^1(\Omega)$ of ansatz classes and penalization strengths $\lambda_n\sim n^\sigma$. Further, we denote the minimizers of the energies $E_{\lambda_n}$ over $V_n$ by $v^\ast_n\in V_n$. It is our goal to choose $\sigma\in\mathbb R$ in such a way that the upper bound of $\lVert v_n^\ast - u^\ast \rVert_{H^1(\Omega)}$ in~\eqref{eq:generalEstimate} decays with the fastest possible rate. Neglecting constants, the bound evaluates to
\begin{align}\label{eq:estimate1}
    \lVert v_n^\ast - u^\ast \rVert_{H^1(\Omega)} \precsim \sqrt{\frac1{\alpha_{\lambda_n}} \inf_{ v\in V_n} \lVert v - u_{\lambda_{n}} \rVert_{a_{\lambda_n}}^2 } + \lambda_n^{-1}.
\end{align}
We can assume without loss of generality that $\sigma>0$ and hence $\lambda_n\ge1$, because otherwise the upper bound will not decrase to zero. Note that in this case we have $\alpha_{\lambda_n}\ge \alpha_1>0$ and hence the we obtain
\begin{align}\label{eq:estimate}
    \lVert v_n^\ast - u^\ast \rVert_{H^1(\Omega)} \precsim
    \sqrt{ \inf_{v\in V_n} \left( \lVert \nabla(v - u_{\lambda_{n}}) \rVert_{L^2(\Omega)}^2 + n^\sigma \lVert  v - u_{\lambda_{n}} \rVert_{L^2(\partial\Omega)}^2 \right) }  + n^{-\sigma}.
\end{align}
Here, the trade off in choosing $\sigma$ and therefore $\lambda_n$ too large or small is evident. We discuss the implications of this upper bound in three different scenarios.

\paragraph{Approximation rates with zero boundary values} 
Consider the case where there is an element \(v_n\in V_n\cap H^1_0(\Omega)\) such that
    \[ \lVert v_n - u^\ast \rVert_{H^1(\Omega)} \precsim n^{-r}. \]
Using the Euler-Lagrange equations $a_\lambda(u_\lambda, \cdot) = f(\cdot)$ and $a(u^\ast, \cdot) = f(\cdot)$ we can estimate 
\begin{align*}
    \frac12\inf_{v\in V_n} \lVert v - u_\lambda \rVert_{a_\lambda}^2 & =  \inf_{v\in V_n} E_\lambda(v) - E_\lambda(u_\lambda) \le  \inf_{v\in V_n\cap H^1_0(\Omega)} E_\lambda(v) - E(u_\lambda) \\ & \le \inf_{v\in V_n\cap H^1_0(\Omega)} E(v) - E(u^\ast) = \frac12 \inf_{v\in V_n\cap H^1_0(\Omega)} \lVert v - u^\ast \rVert_{a}^2 \precsim n^{-2r}
\end{align*}
independently of \(\lambda\).
Hence, the estimate \eqref{eq:estimate1} yields
\[ \lVert v_n^\ast - u^\ast \rVert_{H^1(\Omega)} \precsim \sqrt{n^{-2r}} + \lambda_n^{-1} \precsim n^{-r}, \]
whenever \(\lambda_n\succsim n^r\). Note that in this case, no trade off in \(\lambda\) exists and the approximation rate with zero boundary values can always be achieved up to optimization. However, the curvature \(\alpha_{\lambda_n}\) of \(E_{\lambda_n}\) increases with \(\lambda_n\). Thus, it seems reasonable to choose \(\lambda_n\) as small as possible, i.e., \(\lambda_n\sim n^{r}\). Approximation rates with zero boundary values have not been established so far for neural networks to the best of our knowledge. 

\paragraph{Approximation error of $u_\lambda$ independent of $\lambda$}
Now, we consider the case, without an approximation rate with exact zero boundary values, but where the sequence $(V_n)_{n\in\mathbb N}$ of ansatz classes admits approximation rates in both $H^1(\Omega)$ and $L^2(\partial\Omega)$. More precisely, we assume that
there are real numbers $s\ge r > 0$ such that for every (sufficiently big) \(\lambda\) and every $n\in\mathbb{N}$ there is an element $v_n\in V_n$ satisfying
    \begin{equation*}
        \lVert v_n - u_{\lambda} \rVert_{H^1(\Omega)} \le c n^{-r} \quad \text{and} \quad \lVert v_n - u_{\lambda} \rVert_{L^2(\partial\Omega)}\le c n^{-s},
    \end{equation*}
    for some \(c>0\) independent on \(\lambda\).
Then the estimate in~\eqref{eq:estimate} yields
    \[ \lVert v_n^\ast - u^\ast \rVert_{H^1(\Omega)} \precsim \sqrt{n^{-2r} + n^{\sigma-2s}} + n^{-\sigma}. \]
The resulting rate of decay of the upper bound of the error is then 
    \[ \rho(\sigma) = \min\left( \frac12 \min(2r, 2s-\sigma), \sigma\right) = \min\left(r, s-\frac\sigma2, \sigma\right). \]
which is maximised at \(\sigma^\ast = 2s/3\) with a value of 
\begin{equation}\label{eq:rates2}
    \rho^\ast = \min\left(\frac23 s , r\right).
\end{equation}
In this case, the upper bound does not necessarily decay at the same rate as the approximation error, which decays with rate $r$. Note that because $H^1(\Omega)$ embeds into $L^2(\partial\Omega)$ we can assume without loss of generality that $s\ge r$. 

We made the assumption that the approximation rates of $r$ and $s$ holds with the same constant independently of $\lambda$. This is for example the case, if the solutions $u_\lambda$ are uniformly in $\lambda$ bounded in $H^s(\Omega)$ for some $s>1$. 

\paragraph{Approximation rates for $u^\ast$}
Now we want to discuss the case, where we weaken the approximation assumption from above, which is uniformly in $\lambda$. More precisely, we assume that there is a constant $c>0$ and elements $v_n\in V_n$ satisfying
    \begin{equation*}
        \lVert v_n - u^\ast \rVert_{H^1(\Omega)} \le c n^{-r} \quad \text{and} \quad \lVert v_n - u^\ast \rVert_{L^2(\partial\Omega)}\le c n^{-s}. 
    \end{equation*}
By \eqref{eq:estimateDifference} (or equivalently Theorem~\ref{thm:main} with $V=H^1(\Omega)$ and $v=u_\lambda$) and the triangle inequality we have
    \[ \lVert v_n - u_{\lambda} \rVert_{H^1(\Omega)} \le \lVert v_n - u^{\ast} \rVert_{H^1(\Omega)} + \lVert u^\ast - u_{\lambda} \rVert_{H^1(\Omega)} \le c n^{-r} + c' n^{-\sigma} \le \tilde c n^{-\tilde r} \]
and similarly
    \[ \lVert v_n - u_{\lambda} \rVert_{L^2(\partial\Omega)}\le c n^{-s} + c' n^{-\sigma} \le\tilde c n^{-\tilde s}, \]
where \(\tilde r = \min(r, \sigma)\) and \(\tilde s = \min(s, \sigma)\). Hence, we have reduced this case to the previous case and find that the right hand side of~\eqref{eq:estimate} decays at a rate of 
\begin{align}\label{eq:formulaRho}
    \begin{split}
        \rho(\sigma) & = \min(\min(r, \sigma), \min(s, \sigma) - \sigma/2, \sigma) = \min(r, \min(s, \sigma) - \sigma/2) \\
        & = \min(r, \min(s-\sigma/2, \sigma/2)) = \min(r, s-\sigma/2, \sigma/2). 
    \end{split}
\end{align}
This function is maximized at \(\sigma^\ast = s\) with a value of \(\rho^\ast=\min\left(s/2, r\right)\). Like before, we can without loss of generality assume that $s\ge r$. 
\begin{remark}
Note that in general the decay rate $\rho^\ast = \min\left(s/2, r\right)$ of the upper bound~\eqref{eq:estimate1} can be smaller than the approximation rate $r$. This is in contrast to problems with non-essential boundary values for which the error decays proportional to the approximation error by Cea's lemma. We stress that the defect in the decay rate of the right hand side of~\eqref{eq:estimate1} is not an artefact of our computations but in fact sharp.
\end{remark}

Let us now come back to the original problem of the Dirichlet problem~\eqref{equation:prototype}. For a right hand side $f\in H^r(\Omega)$, standard regularity results yield $u^\ast\in H^{r+2}(\Omega)$. Theorem~\ref{thm:quantUA} provides rates for the approximation in $H^s(\Omega)$ for $s\in[0, 1]$, which lead to the following result.

\begin{theorem}[Rates for NN training with boundary penalty]\label{thm:rates}
    Let $\Omega\subset\mathbb{R}^d$ be a bounded domain with $C^{r+1,1}$ boundary for some $r\in\mathbb N$, $f\in H^r(\Omega)$ and assume $A\in C^{r,1}(\Omega,\mathbb{R}^{d\times d})$ is symmetric, uniformly elliptic with ellipticity constant $\alpha>0$ and denote the solution to~\eqref{equation:prototype} by $u^\ast\in H^1_0(\Omega)$. For every \(n\in\mathbb N\), there is a ReLU network with parameter space \(\Theta_n\) of dimension \(\mathcal O(\log_2^2(n^{(r+2)/d})\cdot n)\) such that if \(\lambda_n\sim n^{\sigma}\) for \(\sigma=\frac{2r + 3}{2d}\) one has for any \(\rho <\frac{2r + 3}{4d}\) that 
    \begin{equation}\label{eq:absEstNNsI}
         \lVert u_{\theta_n} - u^\ast \rVert_{H^1(\Omega)} \precsim \sqrt{\delta_n + n^{-2\rho}} + n^{-\rho} \quad \text{for all } \theta_n\in\Theta_n,
    \end{equation}
    where \(\delta_n = E_{\lambda_n}(u_{\theta_n}) - \inf_{\tilde\theta\in\Theta_n} E_{\lambda_n}(u_{\tilde\theta})\).
\end{theorem}
\begin{proof}
For \(\varepsilon>0\) it holds that \(H^{1/2+\varepsilon}(\Omega)\hookrightarrow L^2(\partial\Omega)\). Thus, Theorem \ref{thm:quantUA} guarantees the existence of $u_{\theta_n}$ with $\theta_n\in\Theta_n$ and the claimed number of parameters such that \(\lVert u_{\theta_n} - u^\ast \rVert_{H^1(\Omega)} \precsim n^{-(r+1)/d} = n^{-\tilde r}\) and 
    \[ \lVert u_{\theta_n} - u^\ast \rVert_{L^2(\partial\Omega)} \precsim \lVert u_{\theta_n} - u^\ast \rVert_{H^{1/2+\varepsilon}(\Omega)} \precsim n^{-(r+2-(1/2+\varepsilon))/d} = n^{-\tilde s}, \]
where \(\tilde s = (2r+3-2\varepsilon)/(2d)\). By~\eqref{eq:formulaRho}, the estimate~\eqref{eq:absEstNNsI} holds for
\begin{align*}
    \rho & = \min(\tilde r, \tilde s - \sigma/2, \sigma/2) = \min\left(\frac{r+1}{d}, \frac{2r+3-4\varepsilon}{4d}, \frac{2r+3}{4d}\right) = \frac{2r+3-4\varepsilon}{4d}.
\end{align*}
\end{proof}

\begin{remark}[Adaptation to Smoothness]\label{rem:adaptionToSmoothness2}
The discussion from Remark~\ref{rem:smoothness} carries over to the case of Dirichlet boundary values and boundary penalties, i.e., the error of the deep Ritz method decays at a rate increasing with the smoothness of the problem. This fact can be especially useful in high spatial dimensions, \jm{which is consistent with the empirical findings that the deep Ritz method can be effective in the numerical solution of high dimensional problems~\cite{weinan2018deep}}.
\jm{
Note that also finite element methods can achieve rates increasing with the smoothness of data, however they require the delicate construction of higher order elements.
}
\end{remark}

\jm{
\begin{remark}[Combination with different approximation results]
We focus on the ReLU activation in this section, whereas in practice often other architectures and activation functions are used, see~\cite{weinan2018deep, hennigh2021nvidia}. However, our results from Section 3 can handle arbitrary function classes and hence reduce the computation of error estimates to the computation of approximation bounds. Therefore, they can be combined with other approximation results for neural networks in Sobolev norm including the works of~\cite{guhring2020approximation, xu2020finite, siegel2020approximation, hon2021simultaneous, duan2021analysis,de2021approximation}. 
\end{remark}
}

\begin{remark}[The boundary penalty method for FEM]
The boundary penalty method has been applied in the context of finite element approximations \cite{babuvska1973finite} and studied in terms of its convergence rates in \cite{babuvska1973finite, shi1984convergence, barrett1986finite}. The idea of the finite element approach is analogue to the idea of using neural networks for the approximate solution of variational problems. However, one constructs a nested sequence of finite dimensional vectorspaces \(V_h\subseteq H^1(\Omega)\) arising from some triangulation with fineness \(h>0\) and computes the minimiser \(u_h\) of the penalized energy \(E_\lambda\) over \(V_h\). Choosing a suitable triangulation and piecewise affine linear elements and setting \(\lambda\sim h^{-1}\) one obtains the error estimate 
    \[ \lVert u_h - u^\ast \rVert_{H^1(\Omega)} \precsim h, \]
see \cite{shi1984convergence}. At the core of those estimates lies a linear version of C\'ea's Lemma, which already incorporates boundary values. However, the proof of this lemma heavily relies on the fact that the class of ansatz functions is linear and that its minimizer solves a linear equation. This is not  the case for non linear function classes like neural networks. Therefore, our estimates require a different strategy. However, the optimal rate of convergence for the boundary penalty method with finite elements can be deduced from our results. In fact, one can choose a suitable triangulation and an operator \(r_h \colon H^2(\Omega) \to V_h\) such that 
    \[ \lVert r_h u - u \rVert_{H^1(\Omega)} \precsim h \lVert u \rVert_{H^2(\Omega)},   \]
where the approximating functions $u_h$ have zero boundary values as they arise from interpolation. By the general discussion from above for the ansatz classes with  approximation rates with exact zero boundary values, choosing $\lambda\succsim h^{-1}$ yields an error bound decaying like the approximation error $\lVert u_h - u^\ast\rVert_{H^1(\Omega)} \precsim h$.
\end{remark}

\acks{The authors want to thank Luca Courte, Patrick Dondl and Stephan Wojtowytsch for their valuable comments. 
JM acknowledges support by the Evangelisches Studienwerk e.V. (Villigst), the International Max Planck Research School for Mathematics in the Sciences (IMPRS MiS) and the European Research Council (ERC) under the EuropeanUnion’s Horizon 2020 research and innovation programme (grant number 757983). MZ acknowledges support from BMBF within the e:Med program in the SyMBoD consortium (grant number 01ZX1910C) and the Research Council of Norway (grant number 303362).}

\bibliography{references}

\appendix

\section{Universal approximation in Sobolev Topology}

Theorem~\ref{thm:quantUA} is a restatement of the main result in~\cite{guhring2019error}. Since this formulation is not present in~\cite{guhring2019error}, we provide a short proof of it, which explains how it can be deduced from the results in~\cite{guhring2019error}.

\quantUA*
\begin{proof}
    The approximation results in \cite{guhring2019error} are stated for functions with the unit cube \([0, 1]^d\) as a domain. However, by scaling and possibly extending functions to the whole of \(\mathbb R^d\) this implies analogue results for functions defined on bounded Sobolev extension domains \(\Omega\).

    We examine the proof of \cite{guhring2019error} in order to see that the approximating network architectures do not depend on \(s\). Let us fix a function \(u\in W^{k, p}([0, 1]^d)\). In their notation, for \(M\in\mathbb N\), there are functions \((\phi_m)_{m=1, \dots, M^d}\) and polynomials \((p_m)_{m=1, \dots, M^d}\), such that
        \[ \left\lVert u - \sum \phi_m p_m \right\rVert_{W^{s, p}([0, 1]^d)} \precsim \lVert u \rVert_{W^{k, p}([0, 1]^d)} M^{-(k-s)} \]
    and a ReLU network function \(u_M\) with \(N\precsim M^d\log(M^k)\) parameters such that
    \[ \left\lVert \sum \phi_m p_m - u_M \right\rVert_{W^{s, p}([0, 1]^d)} \le c(s) \lVert u \rVert_{W^{k, p}(\Omega)} M^{-(k-s)}. \]
    This follows from the Lemma C.3, C.4 and Lemma C.6 in \cite{guhring2019error} with \(\varepsilon = M^{-k}\). Note that the functions and networks provided by those lemmata do not depend on \(s\), which is evident as the estimates are first shown for \(s = 0, 1\) and then generalised through interpolation. Now, by the triangle inequality, we have
    \[ \lVert u - u_M \rVert_{W^{s, p}([0, 1]^d)} \le \tilde c(s) \lVert u \rVert_{W^{k, p}([0, 1]^d)} M^{-(k-s)}. \]
    Now the claim follows by choosing \(M\sim n^{1/d}\). The additional square of the logarithm appears since they are considering networks with skip connections and those are then expressed as networks without skip connections, see also Corollary 4.2 in \cite{guhring2019error}.
\end{proof}
\end{document}